\documentclass[letterpaper, 10pt, twocolumn]{ieeeconf}      % Use this line for a4
\IEEEoverridecommandlockouts                              % This command is only
\usepackage{amsmath,mathrsfs,amsfonts,amssymb,graphicx,epsfig}
 \usepackage{amsthm}
\usepackage{subcaption}
\usepackage{color,multirow,rotating}
\usepackage{algorithm,algpseudocode,algorithmicx}
\usepackage{cite,url,framed,bm,balance}

\newtheorem{theorem}{Theorem}
\newtheorem{corollary}[theorem]{Corollary}

\newtheorem{remark}{Remark}

\usepackage{tikz}
\usetikzlibrary{calc,trees,positioning,arrows,chains,shapes.geometric,decorations.pathreplacing,decorations.pathmorphing,shapes, matrix,shapes.symbols}

\newcommand{\differential}{{\rm{d}}}
\renewcommand{\det}{{\mathrm{det}}}

\newcommand{\blkdiag}{\mathrm{blkdiag}}

\allowdisplaybreaks

\title{\LARGE\textbf{
Boundary and Taxonomy of Integrator Reach Sets}
}
\author{Shadi Haddad, Abhishek Halder% <-this % stops a space
\thanks{Shadi Haddad and Abhishek Halder are with the Department of Applied Mathematics, University of California, Santa Cruz, CA 95064, USA,
        {\tt\small{\{shhaddad,ahalder\}@ucsc.edu}}%
}}

\begin{document}

\maketitle
\pagenumbering{arabic}

%%%%%%%%%%%%%%%%%%%%%%%%%%%%%%%%%%%%%%%%%%%%%%%%%%%%%%%%%%%%%%%%%%%%%%%%%%%%%%%%
\begin{abstract}
Over-approximating the forward reach sets of controlled dynamical systems subject to set-valued uncertainties is a common practice in systems-control engineering for the purpose of performance verification. However, specific algebraic and topological results for the geometry of such sets are rather uncommon even for simple linear systems such as the integrators. This work explores the geometry of the forward reach set of the integrator dynamics subject to box-valued uncertainties in its control inputs. Our contribution includes derivation of a closed-form formula for the support functions of these sets. This result, then enables us to deduce the parametric as well as the implicit equations describing the exact boundaries of these reach sets. Specifically, the implicit equations for the bounding hypersurfaces are shown to be given by vanishing of certain Hankel determinants. Finally, it is established that these sets are semialgebraic as well as translated zonoids. Such results should be useful to benchmark existing reach set over-approximation algorithms, and to help design new algorithms for the same.
\end{abstract}

%%%%%%%%%%%%%%%%%%%%%%%%%%%%%%%%%%%%%%%%%%%%%%%%%%%%%%%%%%%%%%%%%%%%%%%%%%%%%%%%

\section{Introduction}\label{sec:introduction}
Consider the forward reach set $\mathcal{R}$ for the integrator dynamics in Brunovsky normal form with $d$ states and $m$ inputs given by
\begin{align}
\dot{\bm{x}} = \underbrace{\blkdiag\left(\bm{A}_{1},\hdots,\bm{A}_{m}\right)}_{=:\bm{A}}\bm{x} + \underbrace{\blkdiag\left(\bm{b}_{1},\hdots,\bm{b}_{m}\right)}_{=:\bm{B}}\bm{u},    
\label{IntegratorODE}    
\end{align}
with a relative degree vector $\bm{r}=(r_1,\hdots,r_m)^{\top}\in\mathbb{N}^{m}$ (i.e., an $m\times 1$ vector of natural numbers) where $r_{1}+\hdots + r_{m} = d$. In (\ref{IntegratorODE}), the symbol $\blkdiag\left(\cdot\right)$ denotes a block diagonal matrix whose arguments constitute its diagonal blocks, and 
\begin{align}
\boldsymbol{A}_{j}:=\left(\mathbf{0}_{r_{j} \times 1}\left|\boldsymbol{e}_{1}^{r_{j}}\right| \boldsymbol{e}_{2}^{r_{j}}|\ldots| \boldsymbol{e}_{r_{j}-1}^{r_{j}}\right), \quad \boldsymbol{b}_{j}:=\boldsymbol{e}_{r_{j}}^{r_{j}},
\label{DefBlocks}    
\end{align}
for $j=1,\hdots,m$. In (\ref{DefBlocks}), the notation $\mathbf{0}_{r_{j} \times 1}$ stands for the column vector of zeros of size $r_{j} \times 1$. For $k\leq\ell$, we use $\bm{e}_{k}^{\ell}$ to denote the $k$th standard basis (column) vector in $\mathbb{R}^{\ell}$.

As a concrete example, consider $d=5$, $m=2$, and $\bm{r}=(2,3)^{\top}$. Then \eqref{IntegratorODE} takes the form
\begin{align}
\dot{\bm{x}} = \left(\begin{array}{@{}cc|ccc@{}}
0 & 1 & 0 & 0 & 0\\
0 & 0 & 0 & 0 & 0\\ \hline
0 & 0 & 0 & 1 & 0\\
0 & 0 & 0 & 0 & 1\\
0 & 0 & 0 & 0 & 0\\
\end{array}\right)
\bm{x} + \left(\begin{array}{@{}c|c@{}}
0 & 0\\
1 & 0\\ \hline
0 & 0\\
0 & 0\\
0 & 1
\end{array}\right)\bm{u}.   
\label{Example2block3block}
\end{align}

We suppose that the initial condition $\bm{x}_{0}\in\mathbb{R}^{d}$ for (\ref{IntegratorODE}) is fixed, and the control input $\bm{u}\in\mathcal{U}\subset\mathbb{R}^{m}$ with box-valued input set $\mathcal{U}$. The set-valued uncertainties in the input $\bm{u}$ entails uncertainties in the state evolution for (\ref{IntegratorODE}).

Specifically, the state vector $\bm{x}(t)\in\mathcal{R}\left(\{\bm{x}_{0}\},t\right)\subset\mathbb{R}^{d}$, where the set $\mathcal{R}\left(\{\bm{x}_{0}\},t\right)$ is the forward reach set at time $t$ resulting from the singleton initial condition set $\{\bm{x}_0\}$ subject to (\ref{IntegratorODE}), i.e., the set of states the controlled dynamics (\ref{IntegratorODE}) may reach at time $t$ starting from the initial condition $\bm{x}_{0}$. 

Our motivation behind studying the geometry of integrator reach sets is twofold. \begin{itemize}
    \item[(i)] Integrators are simple linear time invariant systems that appear frequently in the literature (see e.g., \cite{maksarov1996state,AROCManual}, \cite[Ch. 3 and 4]{kurzhanski2014dynamics}) to benchmark the performance of reach set over-approximation algorithms. In this context, the lack of knowledge of exact geometry of these sets implies the lack of ground truth against which the conservatism of different over-approximation algorithms can be quantified. Consequently, one has to contend with graphical or statistical performance results. The present work remedies this by deriving the \emph{exact} geometry of integrator reach sets. 
    \item[(ii)] Integrator dynamics in Brunovsky normal form also arise from differentially flat nonlinear systems \cite{fliess1995flatness,murray1995,fliess1999lie}, which in turn appear frequently in vehicular dynamics and control applications. Addressing safety and collision avoidance for such systems in the presence of set-valued uncertainties, then requires computing their (typically nonconvex) reach sets. Having an analytical handle of the corresponding integrator reach sets in the normal coordinates, together with the knowledge of coordinate transforms, could enable the design of novel reach set computation algorithms for this class of nonlinear systems.  
\end{itemize}

It is known \cite{varaiya2000reach} that the set $\mathcal{R}\left(\{\bm{x}_{0}\},t\right)$ is compact and convex. Our prior work \cite{haddad2020convex} investigated the convex geometry of the integrator reach set for \emph{single input case} under the assumption that the set $\mathcal{U}$ is a \emph{symmetric interval} of the form $[-\mu,\mu]\subset\mathbb{R}$. The development herein generalizes that work--we now allow multiple generic interval-valued inputs, i.e., $\mathcal{U}$ is a box of the form 
\begin{align}
\left[\alpha_1,\beta_1\right]\times \left[\alpha_2,\beta_2\right]\times \hdots \left[\alpha_m,\beta_m\right]\subset\mathbb{R}^{m}.
\label{BoxInputSet}    
\end{align}
The box \eqref{BoxInputSet} can be seen as an over-approximation of arbitrary compact   $\mathcal{U}\subset\mathbb{R}^{m}$ with $\alpha_j := \underset{\bm{u}\in\mathcal{U}}{\min} \; u_{j}, \quad \beta_j := \underset{\bm{u}\in\mathcal{U}}{{\max} \; u_{j}}$. The reach set resulting from \eqref{BoxInputSet} will be an over-approximation of that resulting from such compact $\mathcal{U}$.
%, that is, $\alpha_{j}$ and $\beta_{j}$ are the component-wise minimum and maximum, respectively, of the input vector.

With this more generic set up, we obtain exact descriptions of the boundary $\partial\mathcal{R}\left(\{\bm{x}_{0}\},t\right)$, which are novel results. Furthermore, we show that the convex set $\mathcal{R}$ is semialgebraic, translated zonoid, and not a spectrahedron.

One may describe $\mathcal{R}\left(\{\bm{x}_{0}\},t\right)$ via set-valued operation as
\begin{align}
\mathcal{R}\left(\{\bm{x}_{0}\},t\right) = \exp\left(t\bm{A}\right)\bm{x}_{0} \dotplus \int_{0}^{t} \exp\left(s\bm{A}\right)\bm{B}\mathcal{U}\:\differential s,  
\label{ReachSet}    
\end{align}
with each diagonal block being upper triangular, given by
\begin{align}
\exp(\bm{A}_{j}(t-s)) := \begin{cases}
 \dfrac{(t-s)^{\ell-k}}{(\ell-k)!} & \text{for}\; k\leq\ell,\\
 0 & \text{otherwise}.	
 \end{cases}
\label{diagBlocksOfSTM}	
\end{align}
for $k,\ell = 1,\hdots,r_{j}$, for each $j=1,\hdots,m$.
The second summand in (\ref{ReachSet}) is a set-valued Aumann integral and $\dotplus$ stands for the Minkowski sum. In other words, $\mathcal{R}\left(\{\bm{x}_{0}\},t\right)$ is simply a translation of the set given by the second summand in (\ref{ReachSet}). From here, it is not clear if a more explicit description of $\mathcal{R}$ is possible or not. Moreover, a description at the level of (\ref{ReachSet}) is not very helpful for computation. Even at the analysis level, it is not immediate exactly what kind of convex set the second summand in (\ref{ReachSet}) is. In this paper, we show that these issues can be addressed in surprisingly explicit manner.

%The rest of this paper is structured as follows. 
In Sec. \ref{sec:SupportFunctionSec}, we deduce the support function of the forward reach set of (\ref{IntegratorODE}) subject to box-valued uncertainties in its control input. The development in Sec. \ref{sec:SupportFunctionSec} allows us to derive exact parametric (Sec. \ref{subsec:ParamBoundary}) as well as implicit equations (Sec. \ref{subsec:ImplicitizationSec}) for the boundary of this forward reach set. In Sec. \ref{sec:Classification}, we use the results from Sec. \ref{sec:Boundary} to show that this reach set is semialgebraic but not a spectrahedron. Sec. \ref{sec:Conclusion} concludes the paper.

% As an application of our results, we demonstrate that an exact computation of the integrator reach set leads to an exact computation of the reach set of a nonlinear state feedback linearizable system. We give some illustrative examples.

\section{Support Function of the Integrator Reach Set}\label{sec:SupportFunctionSec} 
The support function $h_{\mathcal{K}}:\mathbb{R}^{d}\mapsto\mathbb{R}$ of a compact convex set $\mathcal{K} \subset \mathbb{R}^{d}$, is given by
\begin{align}
h_{\mathcal{K}}(\bm{y}) := \underset{\bm{x}\in\mathcal{K}}{\sup}\:\{\langle\bm{y},\bm{x}\rangle \mid \bm{y}\in\mathbb{R}^{d}\},
\label{DefSptFn}	
\end{align}
where $\langle\cdot,\cdot\rangle$ denotes the standard Euclidean inner product. Geometrically, $h_{\mathcal{K}}(\bm{y})$ gives the signed distance of the supporting hyperplane of $\mathcal{K}$ with outer normal vector $\bm{y}$, measured from the origin. the distance is negative if an only if $\bm{y}$ points into the open halfspace containing the origin. Furthermore, the supporting hyperplane at $\bm{x}^{\text{bdy}}\in\partial\mathcal{K}$ is $\langle\bm{y},\bm{x}^{\text{bdy}}\rangle = h_{\mathcal{K}}(\bm{y})$, and we can write
\[\mathcal{K} = \big\{\bm{x}\in\mathbb{R}^{d} \mid \langle\bm{y},\bm{x}\rangle \leq h_{\mathcal{K}}(\bm{y})\;\text{for all}\;\bm{y}\in\mathbb{R}^{d}\big\}.\]

% Given matrix-vector pair $(\bm{\Gamma},\bm{\gamma})\in\mathbb{R}^{d\times d}\times\mathbb{R}^{d}$, the support function of the affine transform $\bm{\Gamma}\mathcal{K}+\bm{\gamma}$ is 
% \begin{align}
% h_{\bm{\Gamma}\mathcal{K}+\bm{\gamma}}(\bm{y}) = \langle\bm{y},\bm{\gamma}\rangle + h_{\mathcal{K}}(\bm{\Gamma}^{\top}\bm{y}).
% \label{SptFnAffineTransform}	
% \end{align}

Given a function $f:\mathbb{R}^{d}\mapsto\mathbb{R}\cup\{+\infty\}$, its Legendre-Fenchel conjugate is 
\begin{align}
f^{*}(\bm{y}) := \underset{\bm{x}\in\operatorname{domain}(f)}{\sup}\:\{\langle\bm{y},\bm{x}\rangle - f(\bm{x}) \mid \bm{y}\in\mathbb{R}^{d}\}.
\label{LegFenchelConjugate}	
\end{align}
From (\ref{DefSptFn})-(\ref{LegFenchelConjugate}), it follows that $h_{\mathcal{K}}(\bm{y})$ is the Legendre-Fenchel conjugate of the indicator function
\[\mathbf{1}_{\mathcal{K}}(\bm{x}) := \begin{cases}0 & \text{if}\quad \bm{x}\in\mathcal{K},\\
+\infty &\text{otherwise.}\end{cases}\]
Thus, the support function $h_{\mathcal{K}}(\bm{y})$ uniquely determines the set $\mathcal{K}$. Since the indicator function of a convex set is a convex function, the biconjugate 
\begin{align}
\mathbf{1}_{\mathcal{K}}^{**}(\cdot) = h_{\mathcal{K}}^{*}(\cdot) = \mathbf{1}_{\mathcal{K}}(\cdot).
\label{Biconjugate}    
\end{align}

In (\ref{IntegratorODE}), since the system matrices are block diagonal, for the input set \eqref{BoxInputSet}, each of the $m$ single input integrator dynamics with $r_{j}$ dimensional state subvectors for $j=1,\hdots,m$, are decoupled from each other. Hence $\mathcal{R}\left(\{\bm{x}_{0}\},t\right)\subset\mathbb{R}^{d}$ is the Cartesian product of these single integrator reach sets: $\mathcal{R}_{j}\left(\{\bm{x}_{0}\},t\right) \subset \mathbb{R}^{r_{j}}$ for $j=1,\hdots,m$, i.e.,
\begin{align}
\mathcal{R}=\mathcal{R}_{1}\times\mathcal{R}_{2}\times\hdots\times\mathcal{R}_{m}.
\label{CartesianProduct}	
\end{align}
Notice that we can also write (\ref{CartesianProduct}) as a Minkowski sum $\mathcal{R}_{1} \dotplus \hdots \dotplus \mathcal{R}_{m}$. Since the support function of the Minkowski sum is the sum of support functions, this allows us to write the support function of the reach set at time $t$ as
\begin{align}
&h_{\mathcal{R}\left(\{\bm{x}_{0}\},t\right)}(\bm{y}) = \sum_{j=1}^{m} h_{\mathcal{R}_{j}\left(\{\bm{x}_{0}\},t\right)}(\bm{y}_{j}).
\label{SptFn}    
\end{align}

In the following, it will be helpful to define
\begin{align}
\boldsymbol{\xi}_{j}(s):=\left(\begin{array}{c}
s^{r_{j}-1} /\left(r_{j}-1\right) ! \\
s^{r_{j}-2} /\left(r_{j}-2\right) ! \\
\vdots \\
s \\
1
\end{array}\right),\quad j=1,\hdots, m,
\label{DefXij}    
\end{align}
and 
\[\bm{\zeta}_{j}(t) := \int_{0}^{t}\boldsymbol{\xi}_{j}(s)\differential s \in \mathbb{R}^{r_{j}}.\]
Also, let
\begin{align}
\mu_{j}:=\frac{\beta_{j}-\alpha_{j}}{2}, \quad \nu_{j}:=\frac{\beta_{j}+\alpha_{j}}{2}.
\label{defmujnuj}	
\end{align}
We have the following result.
\begin{theorem}
The support function of the forward reach set $\mathcal{R}(\{\bm{x}_{0}\},t)$ for (\ref{IntegratorODE}) at time $t$, for the input set \eqref{BoxInputSet}, is
\begin{align}
h_{\mathcal{R}(\{\bm{x}_{0}\},t)}\left(\bm{y}\right) = &\sum_{j=1}^{m}\bigg\langle\bm{y}_{j},\exp\left(t\bm{A}_{j}\right)\bm{x}_{j0} + \nu_j\bm{\zeta}_{j}(t)\bigg\rangle \nonumber\\
&+ \mu_j\int_{0}^{t}|\langle \bm{y}_{j},\bm{\xi}_{j}(s) \rangle|\:\differential s,
\label{jthsupportfunction}    
\end{align}
where $\bm{x}_{j0}$ is the $r_{j}$ dimensional subvector of the initial condition $\bm{x}_{0}\in\mathbb{R}^{d}$ corresponding to the relative degree component $r_{j}$, where $j=1,\hdots,m$.
\end{theorem}
\begin{proof}
Since support function is distributive over sum, we have
\begin{align}
h_{\mathcal{R}_{j}(\{\bm{x}_{0}\},t)}\left(\bm{y}_{j}\right) = &\langle\bm{y}_{j},\exp\left(t\bm{A}_{j}\right)\bm{x}_{j0}\rangle\nonumber\\ 
&+ h_{\int_{0}^{t}\bm{\xi}_{j}(s)[\alpha_j,\beta_j]\differential s}(\bm{y}_{j}).
\label{hRj}
\end{align}
The second summand in the right hand side above, is the support function of the set $\int_{0}^{t}\bm{\xi}_{j}(s)[\alpha_j,\beta_j]\differential s$, which we simplify next.

From the definition of support function, we have
\begin{align}
h_{\bm{\xi}_{j}(s)[\alpha_j,\beta_j]}(\bm{y}_{j}) = \underset{u_{j}\in[\alpha_j,\beta_j]}{\sup}\:\langle\bm{y}_{j}, \bm{\xi}_{j}(s)u_{j}\rangle.
\label{hjwithoutintegral}    
\end{align}
Since $\mathcal{U}$ is compact, the optimizer $u_{j}^{\text{opt}}$ in (\ref{hjwithoutintegral}) is a maximizer, and is given by 
\begin{align}
u_{j}^{\text{opt}} = \begin{cases}
\beta_{j} & \text{for} \quad \langle\bm{y}_{j},\bm{\xi}_{j}(s)\rangle \geq 0,\\
\alpha_{j} & \text{for} \quad \langle\bm{y}_{j},\bm{\xi}_{j}(s)\rangle < 0.
\end{cases}
\label{ujopt}    
\end{align}
In terms of the Heaviside step function $H(\cdot)$, we can write 
\begin{align*}
u_{j}^{\text{opt}} &= \alpha_j + (\beta_j - \alpha_j) H(\langle\bm{y}_{j},\bm{\xi}_{j}(s)\rangle)\\
&= \alpha_j + (\beta_j - \alpha_j) \times \frac{1}{2}\left(1 + {\rm{sgn}}\left(\langle\bm{y}_{j},\bm{\xi}_{j}(s)\rangle\right)\right),
\end{align*}
where ${\rm{sgn}}(\cdot)$ denotes the signum function. Therefore,
\begin{align}
\!\!\!h_{\bm{\xi}_{j}(s)[\alpha_j,\beta_j]}(\bm{y}_{j}) &= u_{j}^{\text{opt}} \langle\bm{y}_{j},\bm{\xi}_{j}(s)\rangle\nonumber\\
&= \nu_j\langle\bm{y}_{j},\bm{\xi}_{j}(s)\rangle + \mu_j |\langle\bm{y}_{j},\bm{\xi}_{j}(s)\rangle|,
\label{sptfnjbeforelineartransform}
\end{align}
since $|x| = x\:{\rm{sgn}}(x)$ for any real $x$. 

% Using the property (\ref{SptFnAffineTransform}), from (\ref{sptfnjbeforelineartransform}) we get
% \begin{align}
% &h_{\exp(s\bm{A}_j)\bm{b}_{j}[\alpha_j,\beta_j]}(\bm{y}_{j}) = h_{\bm{b}_{j}[\alpha_j,\beta_j]}(\exp\left(s\bm{A}_j^{\top}\right)\bm{y}_{j})\nonumber\\
% &= \nu_j\langle\bm{y}_{j},\bm{\xi}_{j}(s)\rangle + \mu_j |\langle\bm{y}_{j},\bm{\xi}_{j}(s)\rangle|.
% \label{sptfnjafterlineartransform}
% \end{align}
Applying \cite[Proposition 1]{haddad2020convex} to (\ref{sptfnjbeforelineartransform}), we obtain
\begin{align*}
& h_{\int_{0}^{t}\bm{\xi}_{j}(s)[\alpha_j,\beta_j]\differential s}(\bm{y}_{j}) = \int_{0}^{t} \!\!h_{\bm{\xi}_{j}(s)[\alpha_j,\beta_j]}(\bm{y}_{j})\differential s\\
&= \bigg\langle\bm{y}_{j},\nu_j\bm{\zeta}_{j}(t)\bigg\rangle + \mu_j\int_{0}^{t}|\langle \bm{y}_{j},\bm{\xi}_{j}(s) \rangle|\:\differential s,
\end{align*}
which upon substituting back in (\ref{hRj}), yields
\begin{align}
h_{\mathcal{R}_{j}(\{\bm{x}_{0}\},t)}\left(\bm{y}_{j}\right) = &\bigg\langle\bm{y}_{j},\exp\left(t\bm{A}_{j}\right)\bm{x}_{j0} + \nu_j\bm{\zeta}_{j}(t)\bigg\rangle \nonumber\\
&+ \mu_j\int_{0}^{t}|\langle \bm{y}_{j},\bm{\xi}_{j}(s) \rangle|\:\differential s.
\label{lastbutone}    
\end{align}
Combining (\ref{lastbutone}) with (\ref{SptFn}), we arrive at (\ref{jthsupportfunction}).
\end{proof}
\begin{remark}\label{Remark1SptFnFormula}
The formula (\ref{lastbutone}) has a clear geometric meaning. The integral term is the support function of a scaled zonoid, see \cite[Sec. III.A.1]{haddad2020convex}. The set $\mathcal{R}_{j}(\{\bm{x}_{0}\},t)$ is then a translated zonoid, with a time-varying translation of amount $\exp\left(t\bm{A}_{j}\right)\bm{x}_{j0} + \nu_j\bm{\zeta}_{j}(t)$. Notice the extra translation term compared to \cite[equation (10)]{haddad2020convex}.
\end{remark}
\begin{remark}\label{ReachSetOverapprox}
The formula \eqref{jthsupportfunction} upper bounds the support function of integrator reach set resulting from the same initial condition and arbitrary compact   $\mathcal{U}\subset\mathbb{R}^{m}$ with $\alpha_j := \underset{\bm{u}\in\mathcal{U}}{\min} \; u_{j}, \quad \beta_j := \underset{\bm{u}\in\mathcal{U}}{{\max} \; u_{j}}$. In other words, \eqref{CartesianProduct} will over-approximate the reach set associated with the input set $\mathcal{U}$.  
\end{remark}

% \begin{remark}\label{Remark2SptFnFormula}
% A consequence of (\ref{lastbutone}) is that the $r_{j}$ dimensional volume of the set $\mathcal{R}_{j}$ can be obtained by using the formula given by \cite[Thm. 1]{haddad2020convex} with $\mu$ therein replaced by the half-length of the interval $[\alpha_j,\beta_j]$. From (\ref{CartesianProduct}), the $d$ dimensional volume of the reach set $\mathcal{R}$, then, is simply the product of these $r_{j}$ dimensional volumes for $j=1,\hdots,m$. We refer the readers to \cite[Thm. 5]{haddad2021curious} for details.  
% \end{remark}

%%%%%%%%%%%%%%%%%%%%%%%%%%%%%%%%%%%%%%%%%%
\section{Boundary of the Integrator Reach Set}\label{sec:Boundary}
In this Section, we build on the developments in Sec. \ref{sec:introduction} and \ref{sec:SupportFunctionSec}, to derive the parametric representation of $\bm{x}^{\text{bdy}}\in\partial\mathcal{R}\left(\{\bm{x}_{0}\},t\right)$, the boundary of the forward reach set for integrator dynamics with box-valued uncertainties in control inputs. We then provide a principled approach to deduce the implicit equations of the boundary.

\subsection{Parameterization of the Boundary}\label{subsec:ParamBoundary}
To derive the parametric equations of the boundary of the integrator reach set (\ref{ReachSet}), we start with { \cite[p. 111]{kurzhanski2014dynamics}}
\begin{align}
h_{\mathcal{R}\left(\{\bm{x}_{0}\},t\right)}^{*}\left(\bm{x}^{\textup{bdy}}\right) = 0,
\label{BiconjugateEqualsToZero}	
\end{align} 
where $h_{\mathcal{R}\left(\{\bm{x}_{0}\},t\right)}^{*}$ denotes the Legendre-Fenchel conjugate of the support function (\ref{lastbutone}). That \eqref{BiconjugateEqualsToZero} yields the equations for the boundary, follows from \eqref{Biconjugate}.

At any $\bm{x}^{\textup{bdy}}\in\partial\mathcal{R}\left(\{\bm{x}_{0}\},t\right)$, consider a supporting hyperplane $\langle\bm{y},\bm{x}^{\textup{bdy}}\rangle = h_{\mathcal{R}\left(\{\bm{x}_{0}\},t\right)}(\bm{y})$ with outward normal $\bm{y}\in\mathbb{R}^{d}$. Using (\ref{SptFn}) and (\ref{lastbutone}), the equation (\ref{BiconjugateEqualsToZero}) in our case, becomes
\begin{align}
&\sum_{j=1}^{m}\:\underset{\bm{y}_{j}\in\mathbb{R}^{r_{j}}}{\inf}\!\bigg\{\!\langle-\bm{x}^{\textup{bdy}}_{j}+\exp\left(t\bm{A}_{j}\right)\bm{x}_{j0}+\nu_{j}\bm{\zeta}_j(t),\bm{y}_{j}\rangle \nonumber\\
&\qquad\qquad\qquad\qquad\qquad + \mu_{j}\!\!\int_{0}^{t}\!\!|\langle\bm{y}_{j},\bm{\xi}_{j}(s)\rangle| \:\differential s\!\bigg
\} = 0.
\label{UnconstrainedMinimizationSeparableSum}	
\end{align}
In the Theorem \ref{Thm:ParametricRepresentationOfBoundaryPoint} that follows, we use (\ref{UnconstrainedMinimizationSeparableSum}) to derive the parametric equations for $\bm{x}^{\text{bdy}}\in\partial\mathcal{R}\left(\{\bm{x}_{0}\},t\right)$.
%%%%%%%%%%%%%%%%%%%%%%%%%%%%%%%%%%%%%%Boundary parametrization
\begin{theorem}\label{Thm:ParametricRepresentationOfBoundaryPoint}
Consider the reach set (\ref{ReachSet}) with singleton $\mathcal{X}_{0}\equiv\{\bm{x}_{0}\}$ and $\mathcal{U}$ given by \eqref{BoxInputSet}. 
The fixed vector $\bm{x}_{0}\in\mathbb{R}^{d}$ comprises of subvectors $\bm{x}_{j0}\in\mathbb{R}^{r_{j}}$, where $j=1,\hdots,m$, and the relative degree vector $\bm{r}=(r_1,\hdots,r_m)^{\top}\in\mathbb{N}^{m}$. Define $\{\mu_{j}\}_{j=1}^{m}$ and $\{\nu_{j}\}_{j=1}^{m}$ as in (\ref{defmujnuj}). Let the indicator function $\mathbf{1}_{k\leq \ell}:=1$ for $k\leq \ell$, and $:=0$ otherwise.
Then the components of 
\[\bm{x}^{\textup{bdy}}=\begin{pmatrix}
\bm{x}^{\textup{bdy}}_{1}\\
\bm{x}^{\textup{bdy}}_{2}\\
\vdots\\
\bm{x}^{\textup{bdy}}_{m}	
\end{pmatrix}
\in\partial\mathcal{R}\left(\{\bm{x}_{0}\},t\right),\;\bm{x}^{\textup{bdy}}_{j}\in\mathbb{R}^{r_j}, \; j=1,\hdots,m,\] 
admit parameterization
{\small{\begin{align}
&\bm{x}^{\textup{bdy}}_{j}\left(k\right) =  \sum_{\ell = 1}^{r_{j}}\mathbf{1}_{k\leq \ell}\:\frac{t^{\ell-k}}{(\ell-k)!}\:\bm{x}_{j0}(\ell) +\frac{\nu_{j}\:t^{r_{j}-k+1}}{(r_{j}-k+1)!} \nonumber\\
&\pm \frac{\mu_{j}}{(r_{j}-k+1)!}
\bigg\{(-1)^{r_{j}-1}\:t^{r_{j}-k+1}+2\sum_{q=1}^{r_{j}-1}\!\!(-1)^{q+1}\:s_{q}^{r_{j}-k+1}\bigg\},
\label{ParamRepresentationBoundary}	
\end{align}	}}
where $\bm{x}^{\textup{bdy}}_{j}\left(k\right)$ denotes the $k$\textsuperscript{th} component of the $j$\textsuperscript{th} subvector $\bm{x}^{\textup{bdy}}_{j}$ for $k=1,\hdots,r_{j}$. The parameters $(s_{1}, s_{2}, \hdots, s_{r_{j}-1})$ satisfy $0\leq s_{1} \leq s_{2} \leq \hdots \leq s_{r_{j}-1}\leq t$.
\end{theorem}
\begin{proof}
Each of the $m$ objectives in (\ref{UnconstrainedMinimizationSeparableSum}) contains the absolute value of a polynomial in $s$ that can have at most $r_j-1$ roots and therefore at most  $r_j-1$ sign changes in the interval $[0,t]$. We denote these roots as $s_{1} \leq s_{2} \leq \hdots \leq s_{r_{j}-1}$, and write:
{\small{\begin{align}
&\int_{0}^{t}\!\!\!\!|\langle\bm{y}_{j},\bm{\xi}_{j}(s)\rangle|\:\differential s = \pm\!\int_{0}^{s_{1}}\!\!\!\!\langle \bm{y}_{j},\bm{\xi}_{j}(s)\rangle\:\differential s \mp \!\int_{s_{1}}^{s_{2}}\!\!\!\!\langle \bm{y}_{j},\bm{\xi}_{j}(s)\rangle\:\differential s\nonumber\\
&\qquad\qquad\qquad\quad\quad\quad~~\pm\hdots \pm (-1)^{r_{j}-1}\!\int_{s_{r_{j}-1}}^{t}\!\!\!\!\!\!\!\!\!\langle \bm{y}_{j},\bm{\xi}_{j}(s)\rangle\:\differential s\nonumber\\
&= \langle\bm{y}_{j}, \pm\bm{\zeta}_{j}(0,s_1)\mp\bm{\zeta}_{j}(s_1,s_2)\pm\hdots \pm (-1)^{r_{j}-1}\bm{\zeta}_{j}(s_{r_{j}-1},t)\rangle.
\label{SumOfIntegrals}
\end{align}}}
Expression (\ref{SumOfIntegrals}) holds even if the number of roots in $[0,t]$ is strictly less than $r_j-1$ as in that case, the corresponding summand integrals become zero.

Combining  (\ref{SumOfIntegrals}) and (\ref{UnconstrainedMinimizationSeparableSum}) we obtain
\begin{align}
&\sum_{j=1}^{m}\:\underset{\bm{y}_{j}\in\mathbb{R}^{r_{j}}}{\inf}\langle-\bm{x}^{\textup{bdy}}_{j}+\exp\left(t\bm{A}_{j}\right)\bm{x}_{j0} +\nu_{j}\bm{\zeta}_j(t) \pm\mu_{j}\bm{\zeta}_{j}(0,s_1)\nonumber\\
& \mp\mu_{j}\bm{\zeta}_{j}(s_1,s_2)\pm\hdots\pm\mu_{j} (-1)^{r_{j}-1}\bm{\zeta}_{j}(s_{r_{j}-1},t),\bm{y}_{j}\rangle = 0.
\label{LinearMinimizationyj}	
\end{align}
The left-hand-side of (\ref{LinearMinimizationyj}) is the summation of infimum values of linear functions. Therefore, the right-hand-side of (\ref{LinearMinimizationyj}) can achieve zero if and only if each of those linear function equals to zero, i.e., if and only if
\begin{align}
&\bm{x}^{\textup{bdy}}_{j}=\exp\left(t\bm{A}_{j}\right)\bm{x}_{j0} +\nu_{j}\bm{\zeta}_j(t) \pm\mu_{j}\bm{\zeta}_{j}(0,s_1)\mp\mu_{j}\bm{\zeta}_{j}(s_1,s_2)\nonumber\\
&\qquad\quad\pm\hdots\pm (-1)^{r_{j}-1}\mu_{j}\bm{\zeta}_{j}(s_{r_{j}-1},t).
\label{CoeffEqualToZero}	
\end{align}
Using (\ref{DefBlocks}), (\ref{diagBlocksOfSTM}) and (\ref{DefXij}), we simplify (\ref{CoeffEqualToZero}) to (\ref{ParamRepresentationBoundary}). This completes the proof.
% \hfill\qed
\end{proof}

\noindent The following is an immediate consequence of Theorem \ref{Thm:ParametricRepresentationOfBoundaryPoint}.
\begin{corollary}\label{Corollary:TwoBoundingSurfaces}
The single input integrator reach set $\mathcal{R}_{j}\left(\{\bm{x}_{0}\},t\right) \subset \mathbb{R}^{r_{j}}$ has two bounding surfaces for each $j=1,\hdots,m$. In other words, there exist $p_{j}^{\textup{upper}},p_{j}^{\textup{lower}}:\mathbb{R}^{r_{j}}\mapsto\mathbb{R}$ such that
\[\mathcal{R}_{j}\left(\{\bm{x}_{0}\},t\right) = \{\bm{x}\in\mathbb{R}^{r_{j}}\mid p_{j}^{\textup{upper}}(\bm{x})\leq 0, \; p_{j}^{\textup{lower}}(\bm{x})\leq 0\},\]
with boundary $\partial\mathcal{R}_{j}\left(\{\bm{x}_{0}\},t\right) = \{\bm{x}\in\mathbb{R}^{r_{j}}\mid p_{j}^{\textup{upper}}(\bm{x})=0\}\cup\{\bm{x}\in\mathbb{R}^{r_{j}}\mid p_{j}^{\textup{lower}}(\bm{x})=0\}$. 	
\end{corollary}
\begin{proof}
From (\ref{ParamRepresentationBoundary}), we get two different parametric representations of $\bm{x}_{j}^{\text{bdy}}$ in terms of $(s_{1}, s_{2}, \hdots, s_{r_{j}-1})$. One parametric representation results from the choice of positive sign for the $\pm$ appearing in (\ref{ParamRepresentationBoundary}), and another for the choice of negative sign for the same. The plus (resp. minus) sign recovers $p_{j}^{\textup{upper}}(\bm{x})=0$ (resp. $p_{j}^{\textup{lower}}(\bm{x})=0$) in parametric representation (\ref{ParamRepresentationBoundary}).
\end{proof}
%To illustrate Theorem \ref{ParametricRepresentationOfBoundaryPoint}, 
Let us illustrate the boundary parametrization (\ref{ParamRepresentationBoundary}) for the case $\bm{r} = (r_{1},r_{2})^{\top} =(2,3)^{\top}$. In this case, we get
\begin{figure}[tpb]
        \centering
        \includegraphics[width=0.7\linewidth]{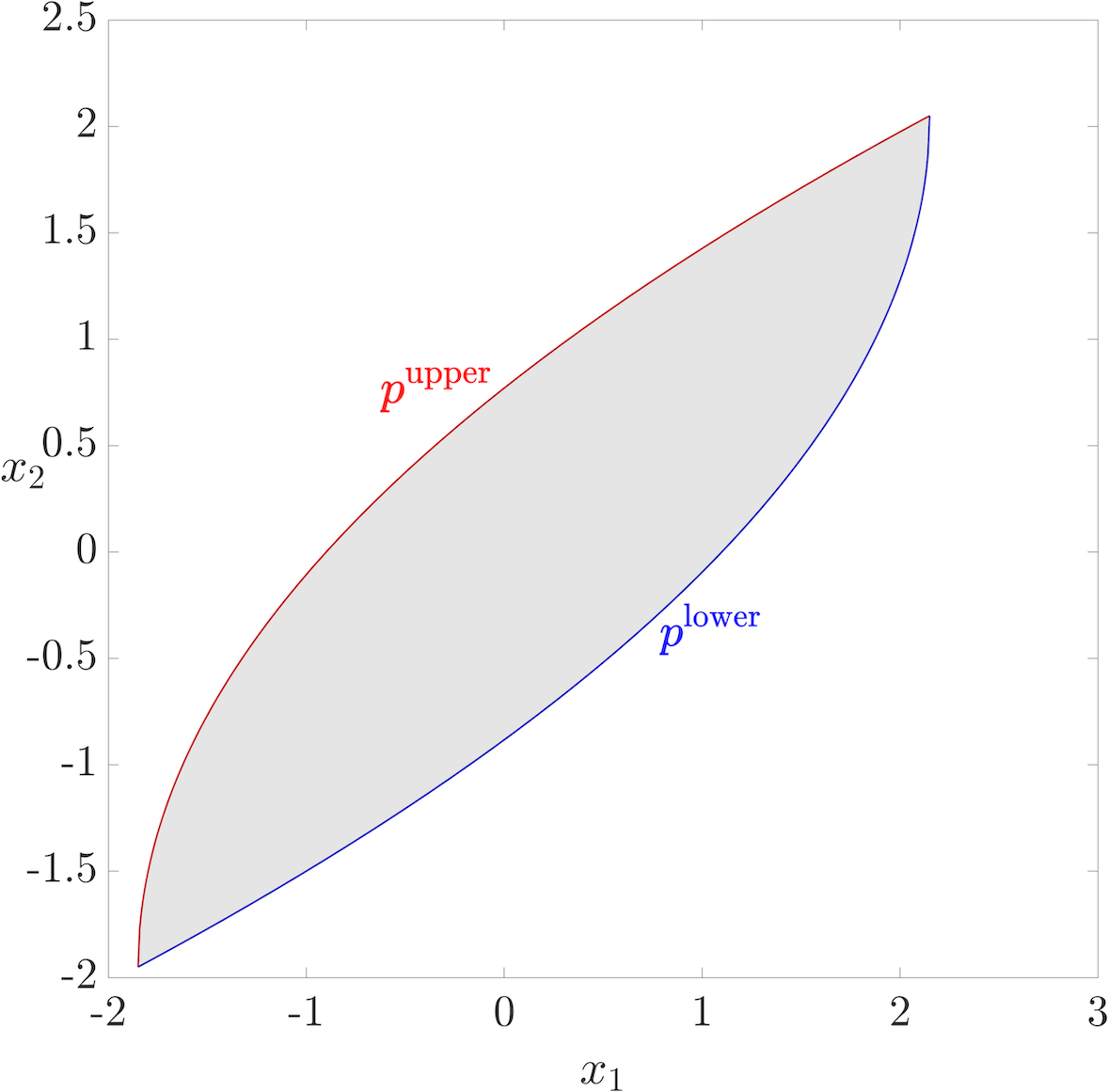}
        \caption{{\small{The integrator reach set $\mathcal{R}(\{\bm{x}_{0}\},t)\subset\mathbb{R}^{2}$ with $d=2$, $m=1$, $\bm{x}_{0}=(0.05,0.05)^{\top}$, $\mathcal{U}\equiv[\alpha,\beta]=[-1,1]$ at $t=2.1$.}}}
\vspace*{-0.2in}
\label{FigSingleInputReachSet2D}
\end{figure}

\begin{figure}[tpb]
        \centering
        \includegraphics[width=0.5\linewidth]{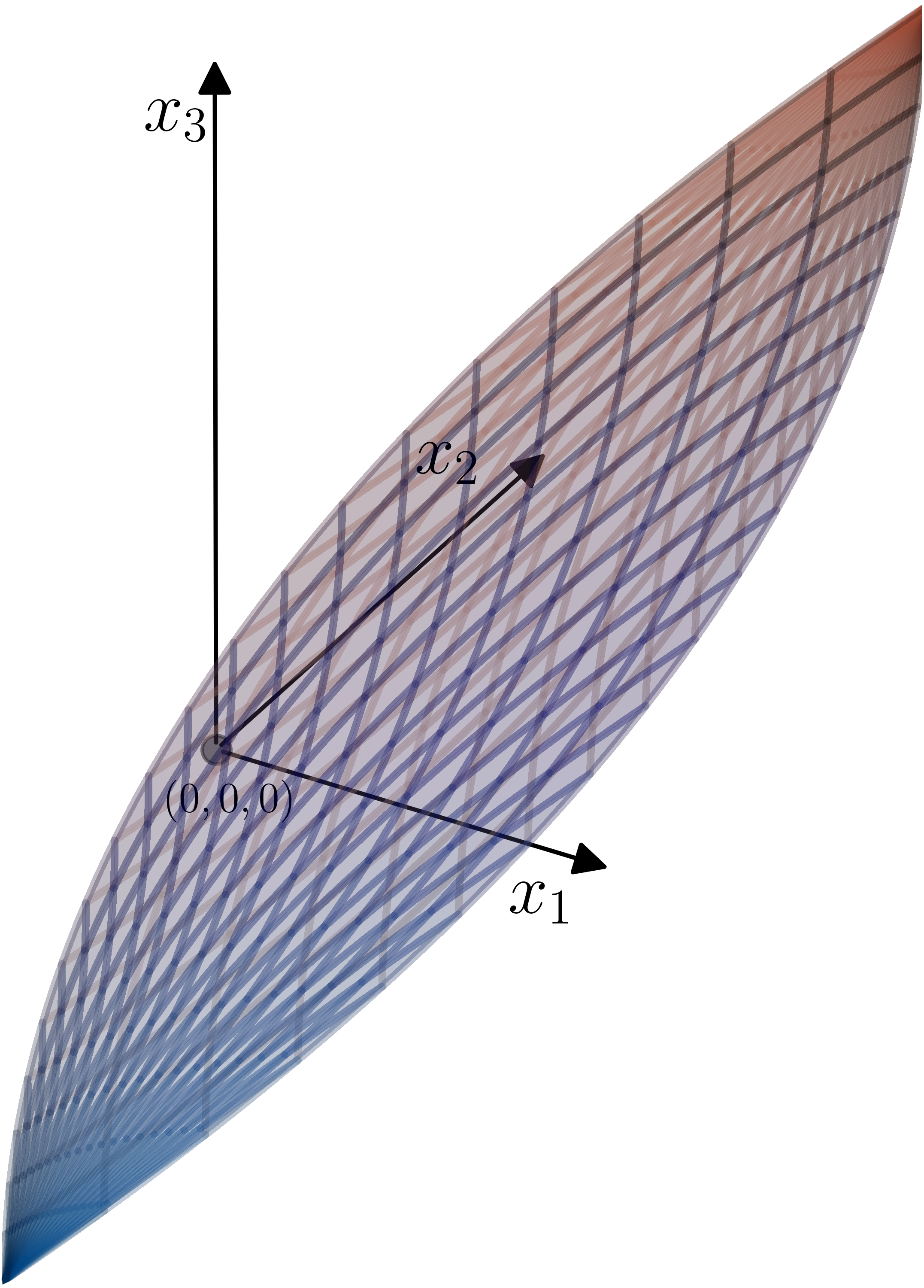}
        \caption{{\small{The integrator reach set $\mathcal{R}(\{\bm{x}_{0}\},t)\subset\mathbb{R}^{3}$ with $d=3$, $m=1$, $\bm{x}_{0} = (0.1,0.2,0.3)^{\top}$, $\mathcal{U}\equiv[\alpha,\beta]=[-1,1]$ at $t=2.1$. The wireframes correspond to the upper and lower surfaces.}}}
\vspace*{-0.2in}
\label{FigSingleInputReachSet3D}
\end{figure}
%\begin{figure}[t]
%\begin{subfigure}{.235\textwidth}
%  \centering
%  \includegraphics[width=0.94\linewidth]{2DIntegratorReachSetSmallIC.png}
%  \caption{{\small{Integrator reach set$\mathcal{R}$ with $d=3$, $\bm{x}_{0}=(0.05,0.05)^{\top}$ at $t=2.1$.}}}
%  \label{sfig2}
%\end{subfigure}\hspace*{0.15in}
%\begin{subfigure}{0.235\textwidth}
%  \centering
%  \includegraphics[width=0.74\linewidth]{TripleIntegratorReachSet3.png}
%  \caption{{\small{The ``almond-shaped" integrator reach set $\mathcal{R}(\{\bm{x}_{0}\},t)\subset\mathbb{R}^{3}$ with $d=3$, $\bm{x}_{0} = (0.1,0.2,0.3)^{\top}$ at $t=2.1$. The wireframes correspond to the upper and lower surfaces.}}}
%  \label{sfig1}
%\end{subfigure}
%\caption{{\small{Double and Triple integrator reach sets for $\mathcal{U}\equiv[\alpha,\beta]=[-1,1]$.}}}
%\label{fig:Boundaries}
%\end{figure}
{\small{\begin{align}
\!\!\begin{pmatrix}
\bm{x}^{\text{bdy}}_{1}(1)\\
\vspace*{-0.05in}\\
\bm{x}^{\text{bdy}}_{1}(2)	
\end{pmatrix}
 \!\!=\!\!\begin{pmatrix}
 \bm{x}_{10}(1) + t\bm{x}_{10}(2) + \nu_{1}(t^{2}/2) \pm\mu_{1}\left(s_{1}^{2}-t^{2}/2\right)\\
 \vspace*{-0.05in}\\
 \bm{x}_{10}(2) + \nu_{1}t \pm \mu_{1}\left(2s_{1}-t\right)	
 \end{pmatrix},
\label{DoubleIntegratorParametric}	
\end{align}}}
and
\begin{align}
\!\!\begin{pmatrix}
\bm{x}^{\text{bdy}}_{2}(1)\\
\vspace*{-0.05in}\\
\bm{x}^{\text{bdy}}_{2}(2)\\
\vspace*{-0.05in}\\
\bm{x}^{\text{bdy}}_{2}(3)	
\end{pmatrix}
 = \begin{pmatrix}
 \bm{x}_{20}(1) + t\bm{x}_{20}(2) + (t^{2}/2)\bm{x}_{20}(3) \\
 + \nu_{2}(t^{3}/6) \pm \mu_{2}\left(t^{3}/6 + 2s_{1}^{3}/6 - 2s_{2}^{3}/6 \right)\\
 \vspace*{-0.05in}\\
 \bm{x}_{20}(2) + t\bm{x}_{20}(3) + \nu_{2}(t^{2}/2) \pm \mu_{2}\left(t^{2}/2 \right.\\
 \left.+ 2s_{1}^{2}/2 - 2s_{2}^{2}/2\right)\\
 \vspace*{-0.05in}\\
  \bm{x}_{20}(3) + \nu_{2}t \pm \mu_{2}\left(t + 2s_{1} - 2s_{2}\right) 
 \end{pmatrix}.\label{TripleIntegratorParametric} 	
\end{align}
%Taking plus (resp. minus) signs in (\ref{DoubleIntegratorParametric}) and  (\ref{TripleIntegratorParametric}) give the parametric representations of the curves $p_{1}^{\text{upper}} = 0$ (resp. $p_{1}^{\text{lower}}=0$) \cite[Fig. 1(a)]{haddad2020convex} and $p_{2}^{\text{upper}}(\bm{x})=0$ (resp. $p_{2}^{\text{lower}}=0$) respectively. These curves are as in \cite[Fig. 1(a)]{haddad2020convex}, and their union defines $\partial\mathcal{R}_{1}$. We note that the parameterization (\ref{DoubleIntegratorParametric}) appeared in \cite[p. 111]{kurzhanski2014dynamics}.
In (\ref{DoubleIntegratorParametric}), taking plus (resp. minus) signs in each of component gives the parametric representation of the curve $p_{1}^{\text{upper}} = 0$ (resp. $p_{1}^{\text{lower}}=0$). The union of these curves defines $\partial\mathcal{R}_{1}$ and are depicted in Fig. \ref{FigSingleInputReachSet2D}. We note that the parameterization (\ref{DoubleIntegratorParametric}) appeared in{ \cite[p. 111]{kurzhanski2014dynamics}}.

Likewise, in (\ref{TripleIntegratorParametric}), taking plus (resp. minus) signs in each of component gives the parametric representation of the surface $p_{2}^{\text{upper}}(\bm{x})=0$ (resp. $p_{2}^{\text{lower}}=0$). The resulting set $\mathcal{R}_{2}$ is the triple integrator reach set, and is shown in Fig. \ref{FigSingleInputReachSet3D}.  The single input integrator reach set for $\bm{r} = (r_{1},r_{2})^{\top} =(2,1)^{\top}$ is shown in Fig. \ref{FigTwoInputReachSet3D}.

Next, we deduce the implicit equations of the boundary associated with the parameterization (\ref{ParamRepresentationBoundary}).
\subsection{Implicitization of the Boundary}\label{subsec:ImplicitizationSec}
%%%%%%%%%%%%%%%%%%%%%%%%%%%%%%%%%%%%%%%%%%
We recall the concepts of \emph{variety} and \emph{ideal} \cite[Ch. 1]{cox2013ideals}.
Let $p_1,\hdots,p_n\in\mathbb{R}[x_1,\hdots,x_d]$, the vector space of real-valued $d$-variate polynomials. The (affine) {variety} $V_{\mathbb{R}[x_1,\hdots,x_d]}(p_1,\hdots,p_n)$ is the set of all solutions of the system $p_1(x_1,x_2,\hdots,x_d)=\hdots = p_n(x_1,x_2,\hdots,x_d)=0$.

Given $p_1,\hdots,p_n\in\mathbb{R}[x_1,\hdots,x_d]$, the set 
\[I:=\bigg\{\sum_{i=1}^{n}\alpha_i p_i \mid \alpha_{1},\hdots,\alpha_{n}\in\mathbb{R}[x_1,\hdots,x_d]\bigg\} \]
is called the {ideal} generated by $p_1,\hdots,p_n$. We write this symbolically as $I=\langle\langle p_1,\hdots,p_n\rangle\rangle$. Intuitively, $\langle\langle p_1,\hdots,p_n\rangle\rangle$ can be thought of as the set of all polynomial consequences of the given system of $n$ polynomial equations in $d$ indeterminates. %Readers can refer to \cite[Ch. 1]{cox2013ideals} for more details on these concepts.

We notice that (\ref{ParamRepresentationBoundary}) gives polynomial parameterizations of the components of $\bm{x}_{j}^{\text{bdy}}$ for all $j=1,\hdots,m$. In particular, for each $k\in\{1,\hdots,r_{j}\}$, the-right-hand-side of (\ref{ParamRepresentationBoundary}) is a homogeneous polynomial in $r_{j}-1$ parameters $(s_{1},s_{2},\hdots,s_{r_{j}-1})$ of degree $r_{j}-k+1$. By polynomial implicitization \cite[p. 134]{cox2013ideals}, the corresponding implicit equations $p_{j}^{\text{upper}}\left(\bm{x}_{j}^{\text{bdy}}\right)=0$ (when fixing plus sign for $\pm$ in (\ref{ParamRepresentationBoundary})) and $p_{j}^{\text{lower}}\left(\bm{x}_{j}^{\text{bdy}}\right)=0$ (when fixing minus sign for $\pm$ in (\ref{ParamRepresentationBoundary})), must define affine varieties $V_{\mathbb{R}[x_1,...,x_{r_{j}}]}(p_{j}^{\text{upper}}), V_{\mathbb{R}[x_1,...,x_{r_{j}}]}(p_{j}^{\text{lower}})$ in $\mathbb{R}\left[x_1,\hdots,x_d\right]$. Theorem \ref{Thm:Implicitization} below provides the implicit equations associated with the parameterization (\ref{ParamRepresentationBoundary}).
%%%%%%%%%%%%%%%%%%%%%
\begin{figure}[t]
        \centering
        \includegraphics[width=0.8\linewidth]{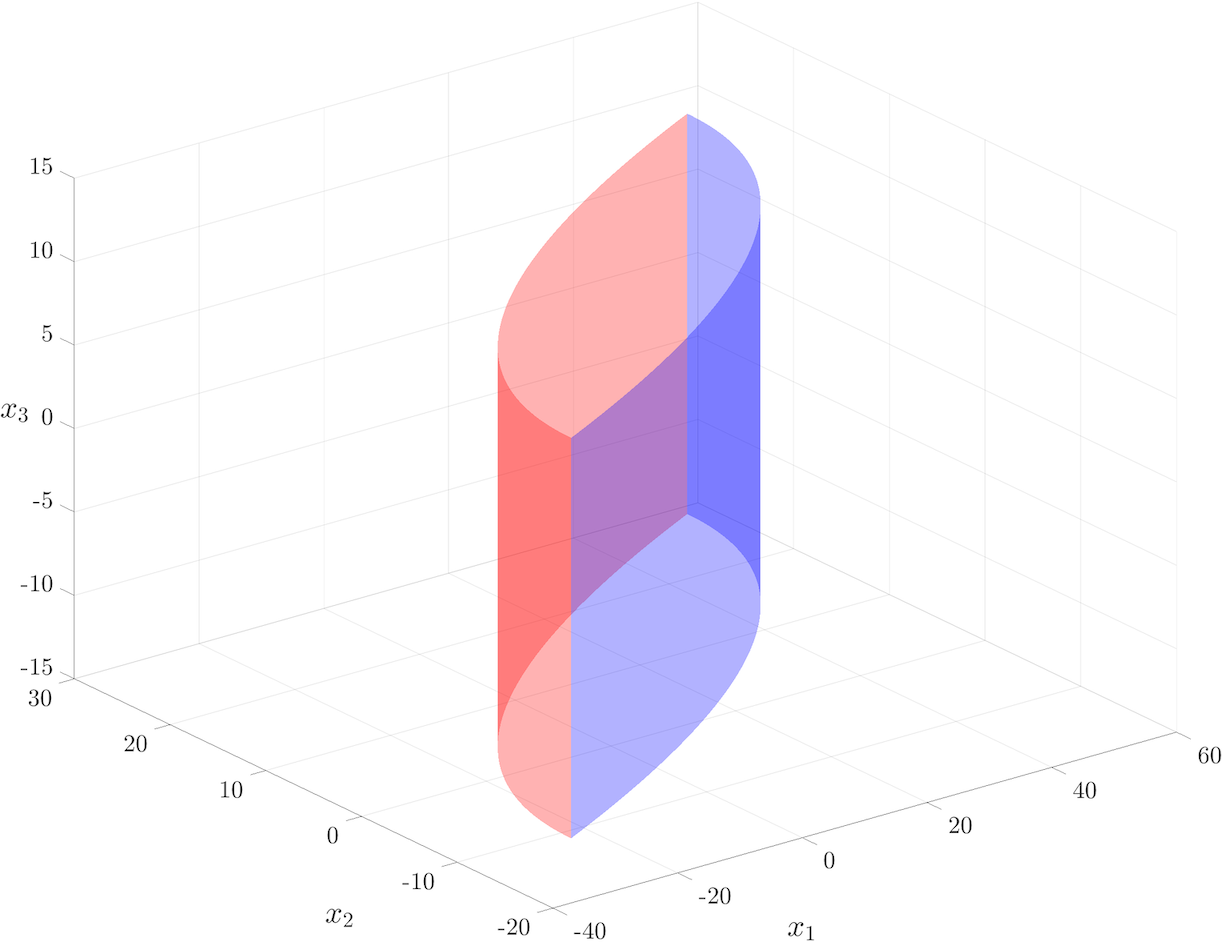}
        \caption{{\small{The integrator reach set $\mathcal{R}(\{\bm{x}_{0}\},t)$ with $d=3$, $m=2$, $\bm{r}=(2,1)^{\top}$, $\bm{x}_{0} = (1,1,0)^{\top}$, $[\alpha_{1},\beta_{1}]=[-5,5]$, $[\alpha_{2},\beta_{2}]=[-3,3]$ at $t=4$.}}}
\vspace*{-0.2in}
\label{FigTwoInputReachSet3D}
\end{figure}
\begin{theorem}\label{Thm:Implicitization}
%%%%%%%%%%%%%%%%%%%%%%%%%
Consider $d$ dimensional integrator reach set (\ref{ReachSet}) resulting from singleton $\mathcal{X}_{0}\equiv\{\bm{x}_{0}\}$ and $\mathcal{U}$ given by \eqref{BoxInputSet}. 
Denote the corresponding single input integrator reach sets $\mathcal{R}_{1}, \hdots,\mathcal{R}_{m}$ as in (\ref{CartesianProduct}). Let $\delta_j := \lfloor \frac{r_{j}-1}{2}\rfloor$ and 
\begin{align}
\rho_{j,k}^{\pm} := &\frac{(r_{j}-k+1)!}{2\mu_{j}}\bigg\{ \bm{x}^{\textup{bdy}}_{j}\left(k\right) -  \sum_{\ell = 1}^{r_{j}}\mathbf{1}_{k\leq \ell}\:\frac{t^{\ell-k}}{(\ell-k)!}\:\bm{x}_{j0}(\ell)\bigg\}\nonumber\\
&-\frac{1}{2}\bigg\{\pm(-1)^{r_{j}-1}\:t^{r_{j}-k+1}+\frac{\nu_{j}}{\mu_{j}}\:t^{r_{j}-k+1} \bigg\},
\label{defrho}	
\end{align}
where $\bm{r}=(r_1,\hdots,r_m)^{\top}\in\mathbb{N}^{m}$ is the relative degree vector, and $\bm{x}_{0}\in\mathbb{R}^{d}$ comprises of subvectors $\bm{x}_{j0}\in\mathbb{R}^{r_{j}}$, $j=1,\hdots,m$. The algebraic hypersurfaces $p_{j}^{\text{upper}}, p_{j}^{\text{lower}} \in \mathbb{R}[x_1,...,x_{r_{j}}]$ bounding $\mathcal{R}_{j}\left(\{\bm{x}_{0}\},t\right)$ have degree $\left(\delta_j+1  \right)\left( r_j-\delta_j \right)$, and are given by the following vanishing Hankel determinant
\begin{align}
\det[A_j^{\pm}(r_j-2\delta_j+i+j)]_{i,j=0}^{\delta_j}=0,
\label{HankelDet}	
\end{align}
where $A^{\pm}_{j}$'s are obtained from
\begin{align}
\sum_{k\geq 0} A^{\pm}_j(k) \tau^k=\exp\left(-\sum_{k=1}^{r_j}\frac{\lambda_{j}^{\pm}(k)}{k}\tau^{k}\right),
\label{ExponentialOfPolySolve}	
\end{align}
and $\lambda_{j}^{\pm}(k):=\rho_{r_j-k+1}^{\pm}$ for $k=1,\hdots,r_j$.
\end{theorem}
\begin{proof}
To ease notation, let us drop the superscript $\pm$ from $\rho_{r_j-k+1}^{\pm},A^{\pm}_{j}$ and $\lambda_{j}^{\pm}(k)$ in (\ref{defrho}). These superscripts can be added back at the end of the proof, without loss of generality.

Following \cite{381335}, we next define the sequence $A_j(k)$ (the sequence is with respect to index $k$) in terms of $s_1,s_2,\dots,s_{r_{j}-1}$ via the generating function (see e.g.,{{\cite[Ch. 1]{wilf2005generatingfunctionology}}})
\begin{align}
F_j(\tau)=\sum_{k\geq 0} A_j(k) \tau^k=\frac{(1-s_1\tau)(1-s_3\tau)\cdots}{(1-s_2\tau)(1-s_4\tau)\cdots}.
\label{GeneratingFn}	
\end{align}
Taking the logarithmic derivative of (\ref{GeneratingFn}), and using the generating functions $(1-s_{q}\tau)^{-1} = \sum_{k\geq 0}\left(s_{q}\tau\right)^{k}$ for all $q=1,\hdots,r_j-1$, yields 
\begin{align}
\!\!\frac{F_{j}^{\prime}(\tau)}{F_{j}(\tau)} = -s_{1}\!\sum_{k\geq 0}\!\left(s_{1}\tau\right)^{k} \!+s_{2}\!\sum_{k\geq 0}\!\left(s_{2}\tau\right)^{k} \!-s_{3}\!\sum_{k\geq 0}\!\left(s_{3}\tau\right)^{k} \!+ \hdots.
\label{LogarithmicDerivative}	
\end{align}
In \eqref{LogarithmicDerivative}, the superscript $^{\prime}$ denotes derivative with respect to $\tau$. Integrating (\ref{LogarithmicDerivative}) with respect to $\tau$, we get
\begin{align}
F_j(\tau) = \exp\left(-\sum_{k=1}^{r_j}\frac{\lambda_{j}(k)}{k}\tau^{k}\right).
\label{ExponentialOfPoly}	
\end{align}
Equating (\ref{GeneratingFn}) and (\ref{ExponentialOfPoly}), we obtain  (\ref{ExponentialOfPolySolve}).

On the other hand, since the generating function (\ref{GeneratingFn}) is a rational function with denominator polynomial of degree $\delta_j$, the Hankel determinant in (\ref{HankelDet}) vanishes\cite{kronecker1881theorie},\cite[p. 5, Lemma III]{salem1983algebraic}. Finally, reverting back the $\lambda_j$'s to the $\rho_j$'s, and inserting back the plus-minus superscripts, result in the desired implicit polynomial, say $\wp(\rho_{j,1}^{\pm},\rho_{j,2}^{\pm},\hdots,\rho_{j,r_j}^{\pm})$, of degree $(\delta_j+1)(r_j-\delta_j)$.
\end{proof}

% \begin{figure*}[ht]
%         \centering
%         \includegraphics[width=0.9\linewidth]{Figures/NonUniqueIntegratorReachSet.pdf}
%         \caption{{\small{\emph{Left}: The reach set in $\mathbb{R}^{2}$ at $t=4$ for the double integrator with two inputs, resulting from singleton initial condition $\{\bm{0}\}$. \emph{Right}: An integrator reach set in $\mathbb{R}^{3}$ at $t=4$ for the relative degree $\bm{r}=(2,1)^{\top}$ with two inputs, resulting from singleton initial condition $\{\bm{0}\}$. In both the plots shown above, the solid boundaries were plotted using Theorem \ref{Thm:ParametricRepresentationOfBoundaryPoint}. The scatter plots were obtained by propagating the states with randomly sampled input trajectories from the sets $\mathcal{U}_{p}$ given by \eqref{pnormball}. Specifically, we divided the time interval $[0,4]$ into four subintervals $[t_{k-1},t_{k}]$ for $k=1,\hdots,4$, $t_{0}:=0,t_{4}:=4$. We used the \emph{piecewise constant} inputs $\bm{u}(t) = \sum_{k=1}^{4}u_{k-1,k}\left(H(t-t_{k-1}) - H(t-t_{k})\right)$ for $t\in[0,4]$, where $H(\cdot)$ denotes the Heaviside step function, and $u_{k-1,k}$ were sampled uniformly at random from the two dimensional $\ell_2$ and $\ell_{\infty}$ unit norm balls (i.e., resampled for each $k=1,\hdots,4$). The plots shown above suggest that with all other parameters being fixed, the integrator reach sets resulting from the input sets $\mathcal{U}_{2}$ and $\mathcal{U}_{\infty}$ are the same, as discussed in Sec. \ref{subsec:Discussions}.}}}
% \vspace*{-0.1in}
% \label{fig:nonunique}
% \end{figure*}

To illustrate Theorem \ref{Thm:Implicitization}, let $d=3$ and $m=1$; the latter allows us to drop the subscript $j$. Then (\ref{HankelDet}) becomes
\begin{align}
\det\left(\begin{bmatrix}
A^{\pm}(1) & A^{\pm}(2)\\
A^{\pm}(2) & A^{\pm}(3)	
\end{bmatrix}
\right) = 0.
\label{HankelDetFor3d}	
\end{align}
In this case, equating (\ref{GeneratingFn}) and (\ref{ExponentialOfPoly}) gives
\begin{align*}
&A^{\pm}(1)=-\lambda^{\pm}(1),\; A^{\pm}(2)=\frac{1}{2}\lambda^{\pm}({1})^{2}-\frac{1}{2}\lambda^{\pm}({2}), \;\\ &A^{\pm}({3})=-\frac{1}{6}\lambda^{\pm}({1})^{3}+\frac{1}{2}\lambda^{\pm}({1})\lambda^{\pm}({2})-\frac{1}{3}\lambda^{\pm}({3}).
\end{align*}
Substituting the above back in (\ref{HankelDetFor3d}) yields the quartic polynomial $\lambda^{\pm}({1})^{4} + 3\lambda^{\pm}({2})^{2}-4\lambda^{\pm}({3})\lambda^{\pm}({1})=0$, which under the mapping $(\lambda^{\pm}(1),\lambda^{\pm}(2),\lambda^{\pm}(3))\mapsto (\rho^{\pm}(3),\rho^{\pm}(2),\rho^{\pm}(1))$, yields degree 4 implicitized polynomial 
\begin{align}
&\wp(\rho^{\pm}(1),\rho^{\pm}(2),\rho^{\pm}(3))\nonumber\\
&=\rho^{\pm}({3})^{4}- 4 \rho^{\pm}({3})\rho^{\pm}({1}) + 3\rho^{\pm}({2})^{2} = 0.
\label{pi3d}	
\end{align} 
For $k=1,2,3$, substituting for the $\rho^{\pm}({1}),\rho^{\pm}({2}),\rho^{\pm}({3})$ in (\ref{pi3d}) from (\ref{defrho}), allows us to recover $p_{j}^{\text{upper}}$ and $p_{j}^{\text{lower}}$.

We can verify (\ref{pi3d}) using elementary algebra by eliminating the parameters $\left(s_{1},s_{2}\right)$ from (\ref{TripleIntegratorParametric}). However, it is practically impossible to derive the implicitization via brute force algebra for $d=4$ or higher.

In summary, (\ref{HankelDet}) gives the implicitization of the bounding hypersurfaces of the single input integrator reach set (up to the change of variables). The implicitization of the hypersurfaces in multi-input case, then, is the Cartesian product of these single input implicit hypersurfaces.

\section{Taxonomy of the Integrator Reach Set}\label{sec:Classification}
In this Section, we address the question: what type of compact convex set the integrator reach set resulting from \eqref{BoxInputSet} is? This question is particularly appealing since several types of convex sets are well-studied in the convex analysis literature. In systems-control literature too, the knowledge of the type of convex set (e.g., semialgebraic, spectrahedron) under investigation, is often leveraged to facilitate analysis, and to design specialized algorithms.

We recall the concept of semialgebraic sets. A set in $\mathbb{R}^{d}$ is called \emph{basic semialgebraic} if it can be written as a finite conjunction of polynomial inequalities and equalities, the polynomials being in $\mathbb{R}\left[x_{1},\hdots,x_{d}\right]$. Finite union of basic semialgebraic sets is called a \emph{semialgebraic set}. A semialgebraic set need not be basic semialgebriac; see e.g., \cite[Example 2.2]{vinzant2020geometry}. Our first taxonomy result is the following.

\begin{theorem}\label{Thm:semialgebraic}
The integrator reach set (\ref{ReachSet}) with $\mathcal{X}_{0}\equiv\{\bm{x}_{0}\}$ and box-valued input set $\mathcal{U}$, is semialgebraic at any time $t$.	
\end{theorem}
\begin{proof}
The implicitization derived in Sec. \ref{subsec:ImplicitizationSec}, by construction, demonstrated that  $p_{j}^{\text{upper}},p_{j}^{\text{lower}}\in\mathbb{R}\left[x_{1},\hdots,x_{r_{j}}\right]$ for all $j=1,\hdots,m$. In words, $p_{j}^{\text{upper}},p_{j}^{\text{lower}}$ are real algebraic hypersurfaces for all $j=1,\hdots,m$. For a different proof that uses the parametric, instead of the implicit representation of the boundary, see \cite[Appendix E]{haddad2021curious}. Alternatively, the statement can be proved by applying Tarski-Seidenberg theorem \cite[Ch. 1]{bochnak2013real} on \eqref{ParamRepresentationBoundary}.
\end{proof}

So far, we have established that the compact convex set $\mathcal{R}\left(\{\bm{x}_{0}\},t\right)$ is semialgebraic (Theorem \ref{Thm:semialgebraic}), and a translated zonoid (Remark \ref{Remark1SptFnFormula}). Convex semialgebraic sets contain a well-known subclass called \emph{spectrahedra}, also referred to as \emph{linear matrix inequality (LMI) representable sets}. Spectrahedra are affine slices of the symmetric positive semidefinite cone. 

The projections of spectrahedra are referred to as \emph{spectrahedral shadows}, which subsume the class of spectrahedra. It is possible that a compact convex set is spectrahedral shadow, i.e., admits LMI representation when lifted to a higher dimensional space, but on its own, not a spectrahedron in the original dimensions. For an example of spectrahedral shadow that is not a spectrahedron, see \cite[p. 23]{helton2010semidefinite}. 

Helton and Vinnikov \cite[Thm. 3.1]{helton2007linear} showed that for a $d$ dimensional compact convex set to be spectrahedron, a necessary condition is \emph{rigid convexity}. A compact semialgebriac set is rigidly convex if the number of intersections made by a generic line passing through an interior point of the set, with its real algebraic boundary is equal to the degree of the bounding algebraic hypersurfaces; see \cite[Sec. 3.1 and 3.2]{helton2007linear}. 

We know from Sec. \ref{subsec:ImplicitizationSec} that the $d$ dimensional single input (i.e., $m=1$) integrator reach set has bounding algebraic surfaces of degree $(\lfloor \frac{d-1}{2}\rfloor + 1)(d-\lfloor \frac{d-1}{2}\rfloor)$. In particular, for $d=2$, Fig. \ref{sfig1} shows that a generic line has $4$ intersections with the bounding real algebraic curves whereas from Theorem \ref{Thm:Implicitization}, we know that $p^{\text{upper}},p^{\text{lower}}$ in this case, are degree $2$ polynomials. 

Likewise, for $d=3$ and $m=1$, Fig. \ref{sfig2} reveals that a generic line has $6$ intersections with the bounding real algebraic surfaces whereas from (\ref{pi3d}), we know that the polynomials $p^{\text{upper}},p^{\text{lower}}$ in this case, are of degree $4$. So the integrator reach set box-valued input set $\mathcal{U}$, is not a spectrahedron.

\begin{figure}[t]
\begin{subfigure}{0.235\textwidth}
  \centering
  \includegraphics[width=0.94\linewidth]{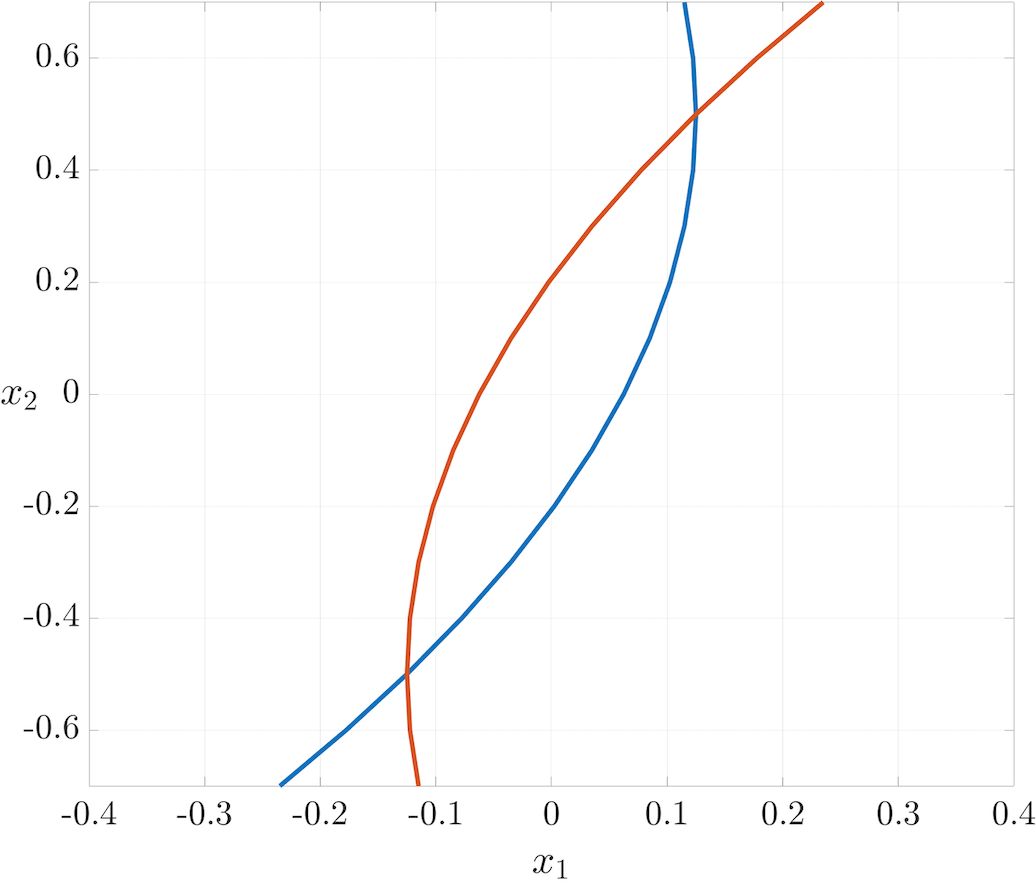}
  \caption{{\small{Real algebraic curves $p^{\text{upper}},\allowbreak p^{\text{lower}}$ for  the double integrator.}}}
  \label{sfig1}
\end{subfigure}\hspace*{0.15in}
\begin{subfigure}{.235\textwidth}
  \centering
  \includegraphics[width=\linewidth]{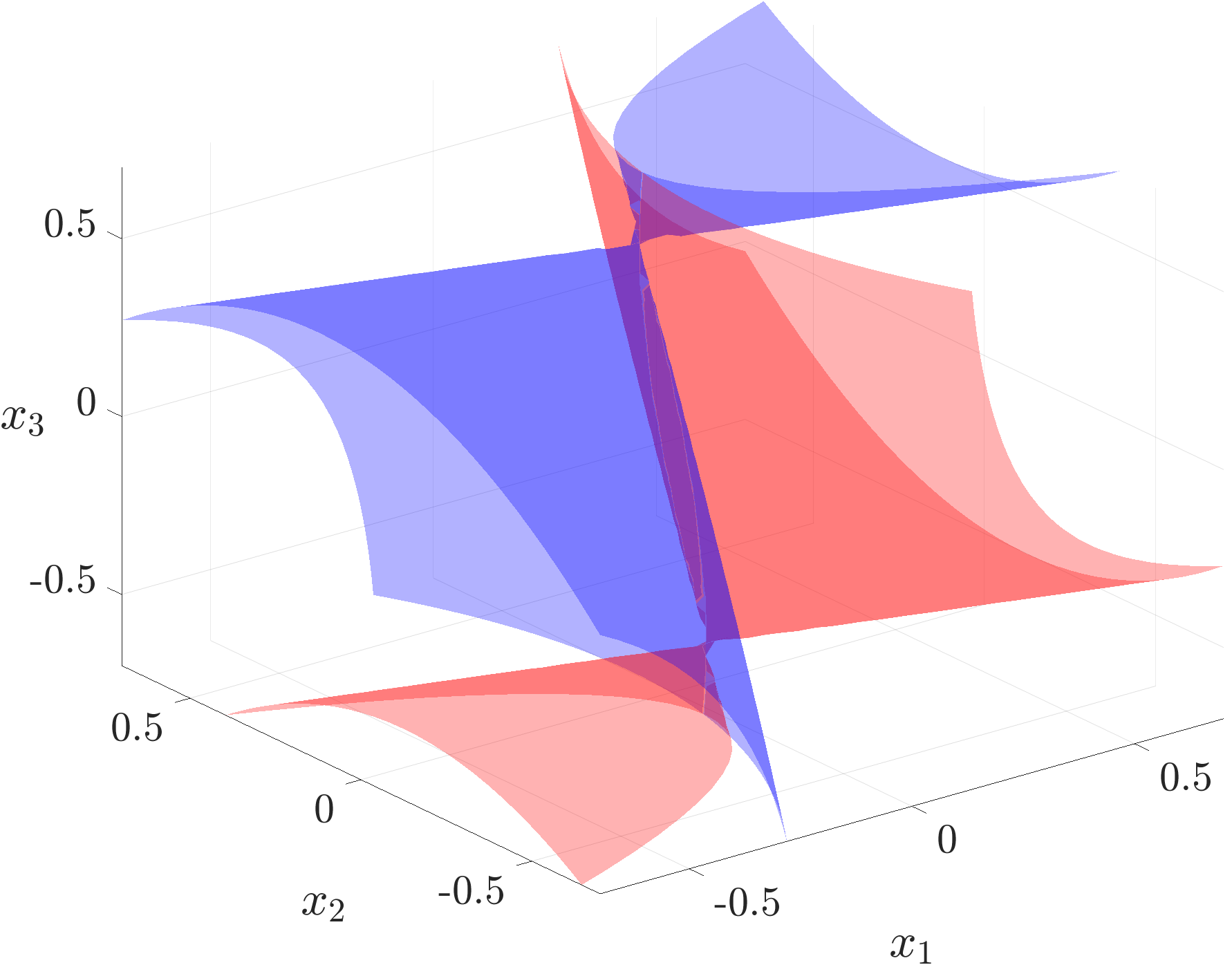}
  \caption{{\small{Real algebraic surfaces $p^{\text{upper}},\allowbreak p^{\text{lower}}$ for the triple integrator.}}}
  \label{sfig2}
\end{subfigure}
\caption{{\small{The bounding polynomials for the double and triple integrator reach sets at $t=0.5$ with $\bm{x}_{0}=\bm{0}$ and $\mu=1$.}}}
\label{fig:Boundaries}
\vspace*{-0.1in}
\end{figure}

\begin{figure}[t]
        \centering
        \includegraphics[width=0.9\linewidth]{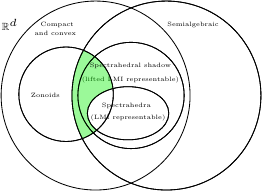}
        \caption{{\small{The summary of the taxonomy for the integrator reach set with box-valued input set $\mathcal{U}$.}}}
\vspace*{-0.2in}
\label{fig:taxonomysummary}
\end{figure}

An interesting question is the following: could this reach set be spectrahedral shadow? Although some calculations show that \emph{sufficient} conditions as in \cite{helton2010semidefinite} do not seem to hold, the authors do not have conclusive answer to this matter. The difficulty here, to the best of the authors' knowledge, is that the necessary conditions for a compact semialgebraic set to be spectrahedral shadow, is not available in the current literature. This makes it hard to falsify this possibility for the integrator reach set.

We graphically summarize our taxonomy results in Fig. \ref{fig:taxonomysummary}; the highlighted region shows where the integrator reach set belongs. Our future work will explore the possibility of further narrowing down the highlighted region.

%%%%%%%%%%%%%%%%%%%%%%%%%%%%%%%%%%%%%%%%%%

% \section{Application: Computing Reach Set of A State Feedback Linearizable System}

%%%%%%%%%%%%%%%%%%%%%%%%%%%%%%%%%%%%%%%%%%

\section{Conclusions and Future Work}\label{sec:Conclusion}

The present paper is part of a research program \cite{haddad2020convex,haddad2021curious} to understand the specific geometry of the integrator reach sets. This is motivated by the fact that integrator reach sets feature prominently for benchmarking the performance of reach set over-approximation algorithms, and that they may be used to over-approximate the forward reach sets of differentially flat nonlinear systems in the normal coordinates (and then numerically map the boundary back in original coorodinates via known coordinate transforms).

In the present paper, we deduced a closed-form formula for the support function of the forward reach set of the integrator dynamics with box-valued uncertainties in the control inputs. Using the support function formula, we then derived analytic expressions of the boundary of these compact convex sets in both parametric and implicit forms. We then delved into the taxonomy question. We argued that this reach set is a translated zonoid, semialgebraic, and not a spectrahedron. 

Our future work will investigate the possibility of further narrowing down the highlighted region in Fig. \ref{fig:taxonomysummary}. In addition, using the ideas presented herein to design algorithms for computing or tightly over-approximating the forward reach sets of differentially flat systems, is a topic of the authors' ongoing research. 
%%%%%%%%%%%%%%%%%%%%%%%%%%%%%%%%%%%%%%%%%%
\balance

\bibliographystyle{IEEEtran}
\bibliography{References.bib}

%
% \begin{thebibliography}{99}

% %% pravin
% \bibitem{varaiya2000reach}
% P. Varaiya, ``Reach set computation using optimal control", in
% M.K. Inan, R.P. Kurshan (eds.), \emph{Verification of Digital and Hybrid Systems, NATO ASI Series (Series F: Computer and Systems Sciences)}, Springer, Vol 170, pp. 323--331, 2000.

% %% Shadi ACC 2020
% \bibitem{haddad2020convex}
% S.~Haddad, and A.~Halder, ``The convex geometry of integrator reach sets", \emph{Proceedings of the 2020 American Control Conference}, pp. 4466--4471, 2020. Extended online version: \\
% \url{https://arxiv.org/pdf/1909.12498.pdf}

% \end{thebibliography}

\end{document}